\documentclass[12pt]{amsart}
\usepackage{amsrefs}
\usepackage{accents}

\usepackage{epsfig,amssymb,amsmath,amsthm,enumerate}
\usepackage[all]{xy}

\newcommand{\xquad}{\rule{8pt}{0pt}}
\newcommand{\uS}{\mathbf{\underaccent{\widetilde{\xquad}}{S}}}
\newcommand{\uSs}{\mathbf{{\underaccent{\widetilde{\xquad}}{S}}_{\mbox{\tiny s}}}}
\newcommand{\uSsp}{\mathbf{{\underaccent{\widetilde{\xquad}}{S}}_{\mbox{\tiny s}}'}}
\newcommand{\uSNU}{\mathbf{{\underaccent{\widetilde{\xquad}}{S}}_{\mbox{\tiny\scshape nu}}}}
\newcommand{\uSos}{{\mathbf{\underaccent{\widetilde{\xquad}}{S}}}_{os}}
\newcommand{\uSosn}{{\mathbf{\underaccent{\widetilde{\xquad}}{S}}}_{os}^{n}}

\newtheorem{lemma}{Lemma}
\newtheorem{definition}{Definition}
\newtheorem{proposition}{Proposition}

\newtheorem{theorem}{Theorem}
\newtheorem{corollary}{Corollary}

\newcommand{\N}{\mathbb{N}}
\newcommand{\meet}{\wedge}
\newcommand{\bigmeet}{\bigwedge}
\newcommand{\join}{\vee}
\newcommand{\irr}{\textup{Irr}}
\newcommand{\con}{\textup{Con}}
\newcommand{\dom}{\textup{dom}}
\newcommand{\pt}{\textup{pt}}
\newcommand{\idl}{\textup{Idl}}
\newcommand{\gr}{\Gamma}
\newcommand{\PDT}{Preduality Theorem}                    
\newcommand{\DAT}{Dual Adjunction Theorem}               
\newcommand{\FDT}{First Duality Theorem}                 
\newcommand{\NUDT}{NU Duality Theorem}                   
\newcommand{\MSDL}{M-Shift Duality Lemma}                
\newcommand{\DET}{Duality and Entailment Theorem}        
\newcommand{\EL}{Entailment Lemma}                       
\newcommand{\BFSDT}{Brute Force Strong Duality Theorem}  
\newcommand{\MSSDL}{M-Shift Strong Duality Lemma}        
\newcommand{\FSDT}{First Strong Duality Theorem}         
\newcommand{\SSDT}{Second Strong Duality Theorem}        
\newcommand{\NUSDT}{NU Strong Duality Theorem}           
\newcommand{\TALM}{Test Algebra Lemma for maps}          

\newcommand{\SP}{ISP}
\newcommand{\ScP}{IS_cP^+}

\newcommand{\Wlog}{Without loss of generality}
\newcommand{\ignore}[1]{\relax}

\newcommand{\downset}{\delimiter"0223379}

\newcommand{\upset}{\delimiter"0222378}
\renewcommand{\chi}{\mathcal{X}}
\newcommand{\calL}{\mathbb{L}}
\newcommand{\calA}{\mathcal{A}}
\newcommand{\bsr}{\mathcal{BSR}}
\newcommand{\pss}{\mathcal{PSS}}

\newcommand{\Tau}{\mathcal{T}}
\renewcommand{\S}{\mathbb{S}}
\newcommand{\bfX}{\mathbf{X}}
\newcommand{\bfY}{\mathbf{Y}}

\begin{document}

\title{The Dual Geometry of Boolean Semirings}

\author[D. J. Clouse]{Daniel J. Clouse}
\address{{\flushleft Daniel J. Clouse}
\newline\indent
Department of Defense}
\email{beckclouse@netzero.net}

\author[F. Guzm\'an]{Fernando Guzm\'an}
\address{{\flushleft Fernando Guzm\'an}
\newline\indent
Dept. of Mathematical Sciences\\
Binghamton University \\ 
Binghamton N.Y., 13902-6000}
\email{fer@math.binghamton.edu}
\urladdr{http://math.binghamton.edu/fer}

\date{January 28, 2008}


\begin{abstract}

  It is well known that the variety of Boolean semirings, which is
  generated by the three element semiring $\S$, is dual to the
  category of partially Stone spaces.  We place this duality in the
  context of natural dualities. We begin by introducing a topological
  structure $\uS$ and obtain an optimal natural duality
  between the quasi-variety $\SP(\S)$ and the category $\ScP(\uS)$.
  Then we construct an optimal and very small structure $\uSos$ that
  yields a strong duality.  The geometry of some of the partially
  Stone spaces that take part in these dualities is presented, and we
  call them ''{\sl hairy cubes}'', as they are $n$-dimensional cubes
  with unique incomparable covers for each element of the cube.  We
  also obtain a polynomial representation for the elements of the
  hairy cube. 

\end{abstract}

\maketitle

\section{Introduction}

Extensions of the concept of a {\sl Boolean Ring} to include semirings
has been done in several different directions.  One source of
diversity are the different definitions of {\sl semiring}.  The other
is how they get connected to Boolean rings.  We will use the concept
of semiring commonly used in formal languages and automata theory,
that is, the only thing missing in order to be a ring is the existence
of additive inverses (see~\cite{Eilenberg} and~\cite{Kuich-Salomaa}).  
As in Guzman~\cite{Guzman}, we will denote by
$\bsr$ the variety generated by the two
2-element semirings, and will call it the variety of Boolean semirings.
It turns out that this variety is also generated 
by a 3-element semiring with carrying set
$S=\{0,h,1\}$, that we denote $\S$. The semiring $\S$ will play a
crucial role in this paper.

In~\cite{Guzman}, following the ideas of Stone~\cite{Stone} in his now
famous {\sl ``Stone representation theorem''}, a duality is
established between the category $\bsr$ of Boolean semirings and the
category $\pss$ of partially Stone spaces.  On the other hand, Clark
and Davey~\cite{Clark-Davey} present a thorough study of natural
dualities between algebraic and topological quasi-varieties.  It is
the goal of this paper to place the duality from~\cite{Guzman} in the
much richer context of~\cite{Clark-Davey}.

A {\sl structured topological space} consists of 
\[ {\bfX} = \langle X; G, H, R, \Tau \rangle \]
where $\langle X,\Tau\rangle$ is a topological space, $G$ is a set of
finitary (total) operations on $X$, $H$ is a set of finitary partial
operations on $X$ and $R$ is a set of finitary relations on $X$.  The
arities of the operations, partial operations, and relations define
the type of ${\bfX}$.  Given a finite discrete structured topological
space ${\bfX}$, we denote by $\ScP(\bfX)$ the category of
closed substructures of non-empty products of copies of $\bfX$.

On one side of the duality we will have the quasi-variety $\calA=\SP(\S)$
generated by $\S$.  On the other side, we will have the category of 
structured topological spaces $\chi=\ScP(\uS)$, generated by some
appropriate structure $\uS=\langle S;G,H,R,\Tau\rangle$ having
$S=\{0,h,1\}$ as underlying set and $\Tau$ the discrete topology. 

We begin by naming three binary relations on $S$: $r_1$, $r_2$, and
$r_3$, and prove the following result.  

\begin{theorem}
  The structure $\uS=\langle{S;\{r_1,r_2,r_3\},\Tau}\rangle$
  yields an optimal natural duality on $\calA$.
\end{theorem}

The proof is in Section~\ref{sec:morphisms}.  The duality is optimal
in the sense that if any one of the relations were to be deleted from
the structure of $\uS$, duality would be lost.  

In the duality $\bsr\leftrightarrows\pss$, a Boolean semiring $A$ is
mapped to the set of prime filters of $A$ (recall that a Boolean
semiring can be viewed as as partially complemented distributive
lattice. See~\cite{Guzman}).  For finite $A$, every prime filter is
the upset of a join-irreducible element.  When $\bfX$ is a closed
substructure of a finite power of $\uS$, we would like to describe the
join-irreducible elements of $\chi(\bfX,\uS)$.  We denote the set of
all of them by $\chi(\bfX,\uS)_J$.  In {\bf Theorem~\ref{thm:ji}},
\stepcounter{theorem}
join-irreducible elements of $\chi(\uS^n,\uS)$ is given; see
Section~\ref{sec:morphisms}.  In~\cite{Clouse} a description of
$\chi(\bfX,\uS)_J$ is given for any $\bfX\in\chi$ that is a closed
substructure of a finite power of $\uS$.  This will appear in a
subsequent paper.  Here we lay the foundation for those results, a
description of the meet semilattice $\chi(\uS^n,\uS)_J$, that we call
the ``{\sl hairy cube}''.  It consists of an $n$-dimensional cube
covered by ``{\sl hairs}''.  More precisely,

\begin{theorem}
  The poset $\chi(\uS^{n},\uS)_J$ consists of two parts:  the ``{\sl
    base}'' $Y^n$ which is an $n$-cube and the ``{\sl hairs}''
  $\chi(\uS^{n},\uS)_J\setminus Y^n$, which are pairwise incomparable.
  Each element of the base is covered by a unique hair.  Each hair
  covers a unique base element.  
\end{theorem}

In {\bf Theorem~\ref{PSSisomorphictohairycube}}
\stepcounter{theorem}%
it is shown that these partial
order properties of $\chi(\uS^{n},\uS)_J$ completely determine it as a
partially Stone space.
We close Section~\ref{sec:ji morphisms} with a polynomial
representation of the join-irreducible elements of
$\chi(\uS^{n},\uS)$.   See {\bf Theorem~\ref{polynomials}}.
\stepcounter{theorem}%

In Section~\ref{sec:makeitstrong} we first establish that the duality
in Theorem~\ref{optdual} is neither a full nor a strong duality.
Then we discuss why this is true and how that duality can easily be
upgraded to a strong duality, following some of the ideas
of~\cite{Clark-Davey}.  Then we show how to construct an optimal and
very small structure $\uSos$ that yields a strong duality on $\calA$.
Despite the fact that $\S$ is not subdirectly irreducible, $\irr(\S)$
is 2, $\uSos$ consists of a single relation $r_2$, and a single
partial operation $\lambda_1$.

\begin{theorem}
Let $\Tau$ denote the discrete topology and
$\uSos=\langle S;\{r_2\},\{\lambda_1\},\Tau\rangle$.  Then $\uSos$ yields an
optimal strong duality on $\calA$.
\end{theorem}

Finally, we discuss why the ``Hairy Cube'' will persist in that strong
duality.

\setcounter{corollary}{5}
\begin{corollary}
  Let $\chi_{os}=\ScP(\uSos)$, then for any $n\in\N$,
  $\chi_{os}(\uSosn,\uS)$  is the n-dimensional Hairy Cube. 
\end{corollary}

The results in this paper and in~\cite{Clouse} greatly expand our
understanding of the dual equivalence between the variety of Boolean
Semirings and the category of Partially Stone Spaces established
in~\cite{Guzman}.  They also a complete a large initial step for
investigating the strong duality we establish between $\calA$
and $\chi_{os}$.

\subsection{Notation}
\setcounter{theorem}{0}
\setcounter{corollary}{0}

\ 

Since some of the arguments in the paper are of an inductive nature,
we need a convenient notation to move back and forth between functions
$S^{n-1}\to S$ and functions $S^n\to S$.

Given $n\in\N$, $\Phi:S^n\to S$, $a\in S$, and $x\in S^{n-1}$ we
denote by $\Phi_a$ the map
\[
\begin{array}{rccc}
  \Phi_a:&S^{n-1}&\to&S \\
         &(x_1,\dots,x_{n-1})&\mapsto&\Phi(a,x_1,\dots,x_{n-1}) \\
\end{array}
\]
and by $\Phi_x$ the map
\[
\begin{array}{rccc}
  \Phi_x:&S&\to&S \\
         &a&\mapsto&\Phi(a,x_1,\dots,x_{n-1}) \\
\end{array}
\]
Given $\psi,\psi',\psi'':S^{n-1}\to S$, we denote by
$(\psi,\psi',\psi'')$ the map
\[
\renewcommand{\arraystretch}{1.2}
\begin{array}{rcccc}
  (\psi,\psi',\psi''):&S^n & \to & S \\
       &(x_1,x_2,\dots,x_n) & \mapsto & \left\{
\renewcommand{\arraystretch}{1.0}
\begin{array}{ll}
  \psi(x_2,\dots,x_n) & \text{ if } x_1=0 \\
  \psi'(x_2,\dots,x_n) & \text{ if } x_1=h \\
  \psi''(x_2,\dots,x_n) & \text{ if } x_1=1 \\
\end{array}\right.
\end{array}
\]

Note that for any $\Phi:S^n\to S$ we have
$\Phi=(\Phi_0,\Phi_h,\Phi_1)$.  In particular, when $n=1$ we may write
any $\Phi:S\to S$ as a triple of elements of $S$;  see, for example,
Lemma~\ref{StoS}.  Since $\Phi_a(x)=\Phi_x(a)$ for any $x\in S^{n-1}$
and $a\in S$, we have $\Phi_x=(\Phi_0(x),\Phi_h(x),\Phi_1(x))$.  When
$\psi:S^{n-1}\to S$ is a term function of $\S$ on $(n-1)$
variables, then $\Psi=(\psi,\psi,\psi)$ is the same term function
viewed as a term function on $n$ variables (with the first one
absent).  We call $\Psi$ {\sl the $n$-ary version of} $\psi$.

\noindent For $1 \leq i \leq n \in \N$ we denote by $\Pi^n_i$ the $i$-th
projection map

\[
\begin{array}{rcccc}
  \Pi^n_i:&S^n & \to & S \\
         &(x_1,\dots,x_n)&\mapsto &x_i\\
\end{array}
\]
For any binary relation $r\subseteq S^2$, we denote by $r^{-1}$ the
{\sl inverse} relation $\{(y,x)|(x,y)\in{r}\}$.

Given a structured topological space $\bfX$, and $\bfY\in\ScP(\bfX)$,
for any operation, partial operation or relation $\lambda$ of the
structure $\bfX$, we denote by $\lambda^\bfY$ the corresponding
operation, partial operation or relation on $\bfY$.

\begin{definition}\label{pcdl}  \cite{Guzman}
A partially complemented distributive lattice is a type
$\langle0,0,0,2,2\rangle$ algebra 
$A=\langle{A;0,h,1,\vee,\wedge}\rangle$ such that
$\langle{A;0,1,\vee,\wedge}\rangle$ is a bounded distributive lattice and
$\langle{[h,1];h,1,\vee,\wedge}\rangle$ is a complemented distributive
lattice, i.e., a Boolean algebra where $[h,1]=\{a\in{A}|h\leq{a}\leq 1\}$.
\end{definition} 

We also wish to note here that it is shown in \cite{Guzman} that 

In any partially complemented distributive lattice one can define a
unary (bar) operation in terms of the complement operation $'$ in
$[h,1]$: 
\[ {\overline{x}}=(x\vee{h})' \]

It satisfies two useful identities can be defined as in the following:

\begin{lemma}\label{barops} \cite{Guzman}
Given a partially complemented distributive lattice
$\langle{A;0,h,1,\vee,\wedge}\rangle$, the bar operation satisfies
\begin{enumerate}[L1)]
\item  $x\vee{\overline{x}}=1$, and 
\item $x\wedge{\overline{x}}=x\wedge{\overline{1}}$. 
\end{enumerate}
\end{lemma}

These two properties characterize partially complemented distributive
lattices. From these results it is derived that
$\bsr$ is dually equivalent to the category of Partially Stone
Spaces, $\pss$, and we denote this dual equivalence as
$\bsr\leftrightarrows\pss$.  

In $\bsr\leftrightarrows{\pss}$, the functor from $\bsr$ 
to $\pss$ takes any partially complemented distributive lattice
$A$ and maps it to $\pt(\idl(A))$, the set of prime filters of $A$. 
These prime filters are
difficult to characterize for arbitrary powers of
$\S$.  For example, if the cardinality of the indexing set is
at least countably infinite, every prime filter has an infinite
descending chain and the existence of of elements of finite support is
unclear.  Hence our current understanding of $\bsr$ 
inherent in this representation is not entirely satisfactory.

\subsection{Natural Dualities}\label{subsec:natural dualities}

\ 

In this paper we will follow very closely the ideas
of~\cite{Clark-Davey} for constructing natural dualities.  The basic
idea is to impose on the carrier $S$ of the semiring $\S$, the
discrete topology together with operations, partial operations and
relations to form a dual topological structure $\uS$ as the generator
of the dual category $\chi$.  More specifically, $\chi=\ScP(\uS)$ is
the category of isomorphic copies, topologically closed substructures
of non-empty products of copies of $\uS$.  Following this
construction, we will have a dual adjunction $\langle
D,E,e,\epsilon\rangle$ between the categories $\calA$ and $\chi$ with
the many desirable properties~\cite{Clark-Davey}*{1.5.3}.  One further
property we desire is that for any $A\in{\calA}$, $A$ is isomorphic to
$ED(A)=\chi(\calA(A,\S),\uS)$.  If the dual adjunction $\langle
D,E,e,\epsilon\rangle$ satisfies this property, it is called a dual
representation of $\calA$ in $\chi$. In this case we say that S yields
a {\sl (natural) duality} on A.  If it is also true that for any
$X\in{\chi}$, $X$ is isomorphic to $DE(X)=\calA((X,\uS),\S)$, we say
that $\uS$ yields a {\sl full duality} on $\calA$.  Thus $\uS$ yields
a full duality on $\calA$ if it yields a duality on $\calA$ which is a
dual equivalence.  Finally, if $\uS$ yields a full duality on $\calA$
and it is injective in $\chi$, $\uS$ is said to yield a {\sl strong
  duality} on $\calA$.

We will construct three dualities, each one coming from a different
topological structure.  In all three of them the algebra side of the
duality will be $\calA=\SP(\S)$.  The first topological structure, $\uS$,
yields an optimal (natural) duality $\chi=\ScP(\uS)\leftrightarrows\calA$.
The second one, $\uSs$, yields a strong duality
$\chi_s=\ScP(\uSs)\leftrightarrows\calA$.  The third one, $\uSos$,
yields an optimal strong duality
$\chi_{os}=\ScP(\uSos)\leftrightarrows\calA$. 
In the first duality,
labeling the appropriate contravariant functors $D$ and $E$, $A$ 
is isomorphic to $ED(A)=\chi(\calA(A,\S),\uS)$.  The situation is displayed
in the diagram below. 
\[\chi\leftrightarrows_{D,E}\calA=\SP(\S)\hookrightarrow{HSP(\S)}=\bsr\leftrightarrows{\pss}\]
Similar remarks hold for the other two dualities. Moreover, in
Corollary~\ref{persist} we show that $\chi_{os}(\uSosn,\uS)=\chi(\uS^n,\uS)$.

Recall that for any $\bfX\in{\ScP(\uS)}$,
$E(\bfX)=\chi(\bfX,\uS)\leqslant{\S^X}$, and 
$DE(\bfX)$ is the set of prime filters of $\chi(\bfX,\uS)$. 
As a result, we desire
that the structure placed on $\uS$ will be sufficient so that we can 
characterize the prime filters of $\chi(\bfX,\uS)$,
and thereby trace their images and the image of $\calA$ in $\pss$.  
In~\cite{Guzman} it is shown that for a finite partially complemented
distributive lattice $\calL$, the partially Stone space $[X,Y]$
corresponding to $\calL$ under the $\bsr\leftrightarrows\pss$ duality,
has $X=\calL_J$ and \mbox{$Y=\{x\in X| x \leq h\}$}.  The
topology of this space is
$\Tau=\{\phi(I)|I\in{\idl(\calL)\}}$ where
$\phi(I)=\{p\in X|p\cap{I}\neq{\emptyset}\}$ for any
$I\in{\idl(\calL)}$.  We call this topology the Stone Space
topology.  In particular, when $\calL=\chi(\uS^n,\uS)$, we identify
the prime filters of $\chi(\uS^n,\uS)$ with its join-irreducible
elements.
A major portion of this paper is devoted to characterizing the join-irreducible
elements of $\chi(\uS^n,\uS)$; we denote the set of such elements by
$\chi(\uS^n,\uS)_J$.

\section{Establishing the Duality and Some Facts About Morphisms}
\label{sec:morphisms}

First of all, we note that $t(x,y,z)=xy+yz+xz$ is a ternary
near-unanimity term on $\S$.  This property will allow us to use many
of the results in~\cite{Clark-Davey} in the construction of the dual
representations that we seek.  We further note that this property
implies that $\bsr$ is a congruence distributive variety and finite products
of $\S$ are skew-free.  Arbitrary products of
$\S$ are known to the authors not to be skew-free, but the
counterexample is outside of the scope of this paper.  

We now define the
first structure $\uS$ that we will show yields a duality on $\calA$.
We will also show that this duality 
is optimal, in the sense that if any single relation were to be deleted 
from $\uS$, duality would be lost.      

\begin{definition}\label{relations}
Define the following subsets of $\S^2$:
\begin{center}

$r_1=\S^2-\{(1,0)\};$

$r_2=\S^2-\{(h,0),(1,0)\};$

$r_3=\S^2-\{(0,1),(h,1),(1,0),(1,h)\}.$
\end{center}
Let $\uS=\langle{S;\{r_1,r_2,r_3\},\Tau}\rangle$ where $\Tau$ is the discrete
topology.
\end{definition}
The following result can be shown by straightforward counting and closure calculations. 
\begin{lemma}\label{subS2}  
  The following is the lattice of subalgebras of $\S^2$:
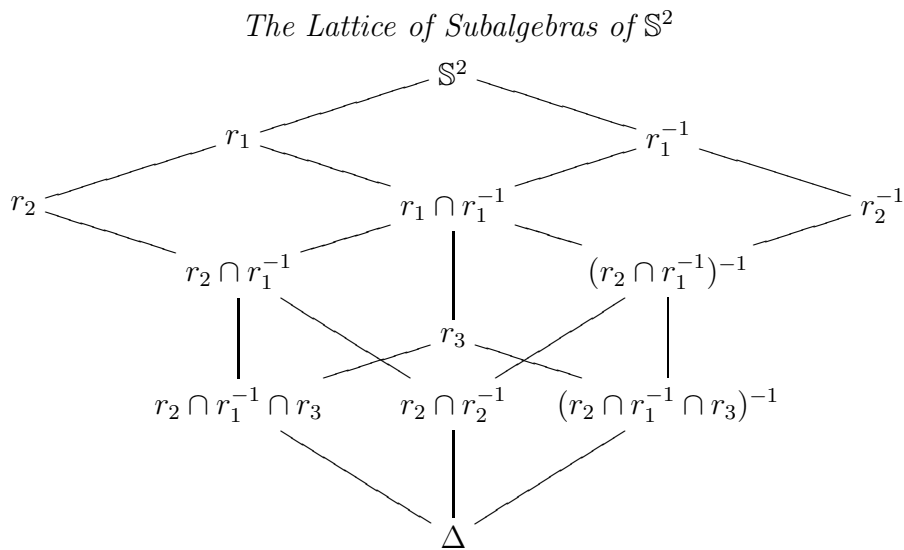
\begin{center}
The Lattice of Subalgebras of $\S^2$

\makebox[3.5in][c]{
\xymatrix @!C@C=-10pt @!R@R=5pt
{
 & & \S^2 \ar@{-}[dl] \ar@{-}[dr] & & \\
& r_1 \ar@{-}[dr] \ar@{-}[dl] & & r_1^{-1} \ar@{-}[dr] \ar@{-}[dl] & \\
r_2 \ar@{-}[dr] & & r_1\cap{r_1^{-1}} \ar@{-}[dl] \ar@{-}[dr]
\ar@{-}[dd] & & r_2^{-1} \ar@{-}[dl] \\
& {r_2\cap{r_1^{-1}}} \ar@{-}[dd] \ar@{-}[ddr] & & ({r_2\cap{r_1^{-1}}})^{-1} \ar@{-}[dd] \ar@{-}[ddl] \\ 
& &r_3 \ar@{-}[dl] \ar@{-}[dr] \\
& r_2 \cap{r_1^{-1}} \cap r_3 \ar@{-}[ddr] & r_2\cap{r_2^{-1}}\ar@{-}[dd] &(r_2\cap{r_1^{-1}}\cap r_3)^{-1} \ar@{-}[ddl] & \\ 
\\ 
& & \Delta & & 
}
}
\vspace{2pt}
\end{center}
\end{lemma}

Consider now the lattice of subalgebras of $\S^2$ coupled with
the following facts:
  \begin{itemize}
  \item if a morphism preserves a binary relation $r$ then it also
    preserves $r^{-1}$, and
  \item if a morphism preserves two k-ary relations $r$ and $s$ then
    it also preserves their intersection $r\cap{s}$.
  \end{itemize}
From this we see that if we use only the set of binary relations
$r_1$, $r_2$, and $r_3$ from Definition~\ref{relations} as structure,
then the structure 
$\uS$ will entail all subalgebras of $\S^2$. Hence by the
\MSDL~\cite{Clark-Davey}*{2.4.2} and the
\NUDT~\cite{Clark-Davey}*{2.3.4} we get the following proposition:

\begin{proposition}\label{duality}  
$\uS$ yields a duality on $\calA$ and $\uS$ is injective in
$\chi$. 
\end{proposition}  

The next lemma determines the elements of 
$\chi(\uS,\uS)\leqslant{\S}^S$. The proof is a 
simple verification of which functions
$f:S\longrightarrow{S}$ preserves the relations of $\uS$.

\begin{lemma}\label{StoS}  
The following is the lattice $\chi(\uS,\uS)$: 

\begin{center}
The Lattice $\chi(\uS,\uS)$
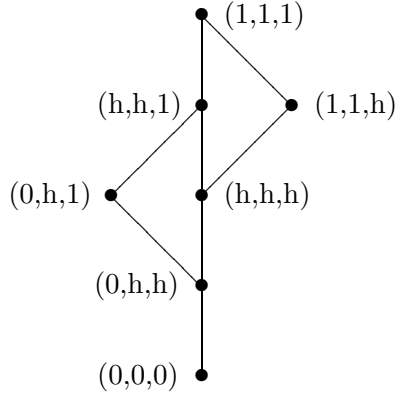
\[
\xy <0mm,0mm>;<1.2mm,0mm>:
(0,0)="a"*{\bullet};
(0,10)="b"**@{-}*{\bullet};
(0,20)="c"**@{-}*{\bullet};
(0,30)="d"**@{-}*{\bullet};
(0,40)="e"**@{-}*{\bullet};
"b";(-10,20)="f"**@{-}*{\bullet};
"f";"d"**@{-};
"c";(10,30)="g"**@{-}*{\bullet};
"g";"e"**@{-};
"a"*!CR{\mbox{\upshape\small(0,0,0)}\ \ };
"b"*!CR{\mbox{\upshape\small(0,h,h)}\ \ };
"c"*!CL{\ \ \mbox{\upshape\small(h,h,h)}};
"d"*!CR{\mbox{\upshape\small(h,h,1)\ \ }};
"e"*!CL{\ \ \mbox{\upshape\small(1,1,1)}};
"f"*!CR{\mbox{\upshape\small(0,h,1)\ \ }};
"g"*!CL{\ \ \mbox{\upshape\small(1,1,h)}};
\endxy
\]
\vspace{0pt}
\end{center}
\end{lemma}

Before we investigate the lattice
$\chi(\uS^n,\uS)_J$, 
we want to show that the duality yielded by $\uS$ on $\calA$ is optimal. 
From the \PDT~\cite{Clark-Davey}*{1.5.2} and the
\FDT~\cite{Clark-Davey}*{2.2.2} we get the following lemma. 

\begin{lemma}\label{lemma:term functions}

  If $\uS'$ is a a structure that yields duality on $\calA$, then the
  finitary term functions on $\S$ must be exactly the morphisms from
  finite powers of $\uS'$ into $\uS'$.  Therefore, for any $n\in\N$
  \[ \chi'({\uS'}^n,\uS') = \chi({\uS}^n,\uS). \]

\end{lemma}

\begin{theorem}\label{optdual}
  The structure $\uS=\langle{S;\{r_1,r_2,r_3\},\Tau}\rangle$
  yields an optimal natural duality on $\calA$.
\end{theorem}

\begin{proof}
Using Lemmas \ref{StoS} and~\ref{lemma:term functions} 
we can show that if
we eliminate $r_2$ or $r_3$ from $\uS$ to form $\uS'$, duality will be lost. The
map $(1,h,1)$ preserves $r_1$ and $r_2$,  but not $r_3$.  The map
$(0,0,h)$ preserves $r_1$ and $r_3$, but not $r_2$.

Now let $\uS'=\langle{S;\{r_2,r_3\},\Tau}\rangle$,
$\chi'=\ScP(\uS')$ and suppose that $\uS'$
yields a duality on $\calA$.  By the \DET~\cite{Clark-Davey}*{2.4.3},
$\{r_2,r_3\}$ must entail $r_1$.  By the \EL~\cite{Clark-Davey}*{2.4.4}, for any
$X\subseteq{S^2}$ and $\alpha\in{\chi'(X,\uS')}$, $\alpha$ must
preserve $r_1$.  Let $X=\{(h,0),(0,1)\}$ and
define $\alpha:X\longrightarrow{\uS}$ by $\alpha((h,0))=1$ and
$\alpha((0,1))=0$.  Then $\alpha$ preserves $r_2$ and $r_3$, but not
$r_1$.  Therefore, we cannot retain duality without $r_1$.
\end{proof}

Upon viewing Lemma~\ref{StoS}, the \PDT\ and the \FDT,
noting that $h$ is a nullary operation of $\S$ and recalling that relations are 
defined pointwise, it is easy to see that Lemma~\ref{construct} holds.
This lemma gives some recursive information about $\chi(\uS^n,\uS)$,
just enough for our needs.

\begin{lemma}\label{construct} Let $n > 1$.
  \begin{enumerate}
  \item Given $\Phi\in\chi(\uS^n,\uS)$ and $x\in S^{n-1}$, we have
    $\Phi_x\in\chi(\uS,\uS)$.
  \item Given $\Phi\in\chi(\uS^n,\uS)$, and $a\in S$, we have
    $\Phi_a\in\chi(\uS^{n-1},\uS)$, and 
    \begin{enumerate}
    \item $\Phi_0\join(\Phi_1\meet h)=\Phi_h$, hence  $\Phi_0\leq\Phi_h$.
    \item $\Phi_0\meet h \leq \Phi_1$.  \\
    Moreover,
    \item
      If $\Phi\leq{h}$ then $\Phi_h=\Phi_1$.  
    \item
      If $\Phi\nleq{h}$ then for every $x\in S^{n-1}$  either
      $\Phi_{0}(x)\leq\Phi_{h}(x)\leq\Phi_{1}(x)$ or
      $\Phi_{0}(x)=\Phi_{h}(x)=1$ and $\Phi_{1}(x)=h$.
    \end{enumerate}
  \item Given $\psi\in\chi(\uS^{n-1},\uS)$, let $\Psi=(\psi,\psi,\psi)$ be
    the $n$-ary version of $\psi$.   We have that
    $\Psi\wedge{h}\wedge{\Pi_1}$,\ 
    $\Psi\wedge{h}\wedge{\overline{\Pi_1}}$,\  
    $\Psi\wedge{\Pi_1}$,\ 
    $\Psi\wedge{\overline{\Pi_1}}\in{\chi(\uS^n,\uS)}$, and 
    \begin{enumerate}
    \item $\Psi\wedge{h}\wedge{\Pi_1}=(0,\psi\wedge{h},\psi\wedge{h})$
    \item $\Psi\wedge{h}\wedge{\overline{\Pi_1}}=(\psi\wedge{h},\psi\wedge{h},
      \psi\wedge{h})$
    \item $\Psi\wedge{\Pi_1}=(0,\psi\wedge{h},\psi)$ 
    \item $\Psi\wedge{\overline{\Pi_1}}=(\psi,\psi,\psi\wedge{h})$ 
    \end{enumerate}
  \end{enumerate}

\end{lemma}

We now have enough results to begin our work characterizing
$\chi(\uS^n,\uS)_J$.  
We begin with a simple but useful Corollary to Lemma~\ref{StoS}.

\begin{corollary}\label{StoSj}
The following diagram is the poset $\chi(\uS,\uS)_J$: 
\begin{center}
The Poset $\chi(\uS,\uS)_J$
\[
\xy <0mm,0mm>;<1.2mm,0mm>:
(0,10)="b"*{\bullet};
(0,20)="c"**@{-}*{\bullet};
"b";(-10,20)="f"**@{-}*{\bullet};
"c";(10,30)="g"**@{-}*{\bullet};
"g";
"b"*!CR{\mbox{\upshape\small(0,h,h)}\ \ };
"c"*!CL{\ \ \mbox{\upshape\small(h,h,h)}};
"f"*!CR{\mbox{\upshape\small(0,h,1)\ \ }};
"g"*!CL{\ \ \mbox{\upshape\small(1,1,h)}};
\endxy
\]
\vspace{0pt}
\end{center}
\end{corollary}

\begin{proposition}\label{prop:ji up down} Let $n > 1$.
  \begin{enumerate}
  \item If $\psi\in\chi(\uS^{n-1},\uS)_J$ and $\Psi=(\psi,\psi,\psi)$ is
    the $n$-ary version  of $\psi$,
    then $\Psi\meet \Pi_1=(0,\psi\meet h,\psi)$ and
    $\Psi\meet\overline{\Pi_1}=(\psi,\psi,\psi\meet h)$ are
    in $\chi(\uS^n,\uS)_J$. 
  \item If  $\Phi\in\chi(\uS^n,\uS)_J$ with $\Phi\leq h$, then there is
    $\psi\in\chi(\uS^{n-1},\uS)_J$ with $\psi\leq h$ such that 
    $\Phi=(0,\psi,\psi)=\Psi\meet\Pi_1$ or
    $\Phi=(\psi,\psi,\psi)=\Psi\meet\overline{\Pi_1}$, where 
    $\Psi=(\psi,\psi,\psi)$ is the $n$-ary version of $\psi$.
  \item If $\Phi\in\chi(\uS^n,\uS)_J$ with $\Phi\nleq h$, then there is
    $\psi\in\chi(\uS^{n-1},\uS)_J$ with $\psi\nleq h$ such that 
    $\Phi=(0,\psi\meet h,\psi)=\Psi\meet\Pi_1$ or
    $\Phi=(\psi,\psi,\psi\meet h)=\Psi\meet\overline{\Pi_1}$, where 
    $\Psi=(\psi,\psi,\psi)$ is the $n$-ary version of $\psi$.
  \end{enumerate}
\end{proposition}

\begin{proof}
  1.  Assume $\psi\in\chi(\uS^{n-1},\uS)$ is  join-irreducible. \\
  Consider first the case
  $\Phi=(0,\psi\meet h,\psi)$, and  suppose
  $\Phi=\Gamma\join\Delta$ with $\Gamma,\Delta\in\chi(\uS^n,\uS)$ and
  $\Gamma,\Delta < \Phi$.  We have
  \[ \Gamma_0=\Delta_0=0, \quad
     \Gamma_h\join\Delta_h=\psi\meet h, \quad
     \Gamma_1\join\Delta_1=\psi. \]
  \Wlog, we have $\Gamma_1=\psi$ and therefore $\Gamma_h < \psi\meet
  h$, i.e. there is $x\in S^{n-1}$ such that 
  \[ 0 \leq \Gamma_h(x) < \psi(x)\meet h \leq \psi(x), h \]
  So, we must have $\Gamma_h(x)=0$. By Lemma~\ref{construct}.1,
  $\Gamma_x\in\chi(\uS,\uS)$ and by Lemma~\ref{StoS}
  $\Gamma_x=(0,0,0)$ which gives $0=\Gamma_1(x)=\psi(x)$, a
  contradiction. \\
  The case  $\Phi=(\psi,\psi,\psi\meet h)$ can be handled in the same
  way, except that instead of going from $\Gamma_1$ to $\Gamma_h$, one
  goes from $\Gamma_0$ to $\Gamma1$.  

  2.  Assume $\Phi\in\chi(\uS^{n},\uS)$ is  join-irreducible, with
  $\Phi\leq h$. \\
  By Lemma~\ref{construct}.2
  $\Phi_h=\Phi_1\in\chi(\uS^{n-1},\uS)$, call it $\psi$, and
  $0\leq \Phi_0\leq \psi$. Moreover, $\psi\leq h$.  By
  Lemma~\ref{construct}.3 we have $(\Phi_0,\Phi_0,\Phi_0)$ and
  $(0,\psi,\psi)$ are in $\chi(\uS^{n},\uS)$, and 
  \[ (\Phi_0,\Phi_0,\Phi_0) \join (0,\psi,\psi) = (\Phi_0,\psi,\psi) =
  \Phi \]
  Therefore, either $\Phi = (0,\psi,\psi)$ or $\Phi =
  (\Phi_0,\Phi_0,\Phi_0) = (\psi,\psi,\psi)$. The join irreducibility
  of $\Phi$ and Lemma~\ref{construct}.3 force $\psi$ to be join
  irreducible. 

  3.  Assume $\Phi\in\chi(\uS^{n},\uS)$ is  join-irreducible, with
  $\Phi\nleq h$. \\
  Using Lemma~\ref{construct}.2,3 we get that
  $(\Phi_0,\Phi_0,\Phi_0\meet h)$ and $(0,\Phi_1\meet h,\Phi_1)$ are
  in  $\chi(\uS^{n},\uS)$.  Using Lemma~\ref{construct}.1 we get
  \[ (\Phi_0,\Phi_0,\Phi_0\meet h)\join(0,\Phi_1\meet h,\Phi_1)=
  (\Phi_0,\Phi_0\join(\Phi_1\meet h),(\Phi_0\meet h)\join \Phi_1) = 
  (\Phi_0,\Phi_h,\Phi_1) = \Phi \]
  Therefore, either $\Phi= (\Phi_0,\Phi_0,\Phi_0\meet h)$ or
  $\Phi=(0,\Phi_1\meet h,\Phi_1)$. In the first case, take
  $\psi=\Phi_0$, in the second case, take $\psi=\Phi_1$. Once again, the
  join irreducibility of $\Phi$ and Lemma~\ref{construct}.3, force
  $\psi$ to be join  irreducible. The fact that $\Phi\nleq h$ yields
  $\psi\nleq h$. 
\end{proof}

From the previous proposition and Corollary~\ref{StoSj},
Corollary~\ref{cor:meet h is ji} follows by induction.

\begin{corollary}\label{cor:meet h is ji}
  If $\Phi\in\chi(\uS^n,\uS)_J$ then $\Phi\meet h\in\chi(\uS^n,\uS)_J$.
\end{corollary}

Combining the different parts of Proposition~\ref{prop:ji up down} we
get the following theorem:

\begin{theorem}\label{thm:ji}
  Let $\Phi\in\chi(\uS^n,\uS)$.
  \begin{enumerate}
  \item\label{thm:ji both}
    $\Phi$ is join irreducible if and only if $\Phi=(0,\psi\meet h,\psi)$
    or $\Phi=(\psi,\psi,\psi\meet h)$ for some join irreducible
    $\psi\in\chi(\uS^{n-1},\uS)$, i.e. if and only if
    $\Phi=\Psi\meet\Pi_1$ or $\Phi =\Psi\meet\overline{\Pi_1}$, where 
    $\Psi=(\psi,\psi,\psi)$ is the $n$-ary version of $\psi$.
  \item\label{thm:ji lt j} 
    When $\Phi\leq h$, $\Phi$ is join irreducible if and only if  
    $\Phi=(0,\psi,\psi)$ or
    $\Phi=(\psi,\psi,\psi)$ for some join irreducible
    $\psi\in\chi(\uS^{n-1},\uS)$ with $\psi\leq h$.
  \item \label{thm:ji nlt j}
    When $\Phi\nleq h$, $\Phi$ is join irreducible if and only if 
    $\Phi=(0,\psi\meet h,\psi)$ or
    $\Phi=(\psi,\psi,\psi\meet h)$ for some join irreducible
    $\psi\in\chi(\uS^{n-1},\uS)$ with $\psi\nleq h$.
  \end{enumerate}
\end{theorem}

\section[Join irreducible morphisms]{The Poset and Polynomial Characterization of the Join-irreducible Morphisms}
\label{sec:ji morphisms}

With Theorem~\ref{thm:ji} at our disposal, we can now proceed to
obtain the poset structure of $\chi(\uS^n,\uS)_J$.  We show that this
poset structure completely determines $\chi(\uS^n,\uS)_J$ as a
partially Stone Space.  Along the way we also obtain a polynomial
representation. 

We are dealing with $\chi(\uS^n,\uS)_J$ with its pointwise partial order
inherited from $\chi(\uS^n,\uS)$.  This is the same as the open set
partial order obtained from the Stone topology.  
When we discuss the properties of the
elements of $\chi(\uS^n,\uS)_J$ covering, being covered or being
incomparable, we will be considering them in the poset
$\chi(\uS^n,\uS)_J$, not in the partially complemented
distributive lattice $\chi(\uS^n,\uS)$. 

\subsection{The Base of the Hairy Cube}\label{the hairy cube}

\begin{proposition}\label{prop:thebase} 
The set $Y^n=\{\Phi\in{\chi(\uS^n,\uS)_J}|\Phi\leq{h}\}$  is
poset isomorphic to $2^n$.
\end{proposition} 

\begin{proof} For $n=1$  see figure of $\chi(\uS,\uS)_J$ given
  in Corollary $\ref{StoSj}$. 

Assume now that there exists a poset isomorphism
$\eta_{n-1}:Y^{n-1}\longrightarrow{2^{n-1}}$.  By
Theorem~\ref{thm:ji}.\ref{thm:ji lt j}
if $\Phi\in{Y^n}$, then we know that
$\Phi=(0,\psi,\psi)$ or $\Phi=(\psi,\psi,\psi)$ for some
$\psi\in{\chi(\uS^{n-1},\uS)_J}$.   Define the following map:
\[
\begin{array}{rccc}
\eta_n:&Y^n             &\longrightarrow&{2^n}\\
     &(0,\psi,\psi)   &\mapsto        &(0,\eta_{n-1}(\psi))\\ 
     &(\psi,\psi,\psi)&\mapsto        &(1,\eta_{n-1}(\psi))\\
\end{array}
\]
The fact that $\eta_n$ is bijective follows immediately from the fact that
$\eta_{n-1}$ is. That $\eta_n$ and its inverse are order preserving is clear
from the definition and the fact that $\eta_{n-1}$ and its inverse are
order preserving.   
\end{proof}

\subsection{The Covers}\label{the covers}

\begin{proposition}\label{prop:covers} Let $n \geq 1$ and
  $\Phi,\Gamma\in\chi(\uS^{n},\uS)_J$. 
  \begin{enumerate}
  \item\label{incomparablecovers} If $\Phi,\Gamma\nleq 
    h$ and $\Phi\neq\Gamma$ then $\Phi,\Gamma$ are incomparable.
  \item\label{uniquecovers} If $\Phi\leq{h}$ then
    there is a unique
    $\widetilde{\Phi}\in{\chi(\uS^n,\uS)_J}$   with
    $\widetilde{\Phi}\nleq{h}$, such that $\widetilde{\Phi}$ covers
    $\Phi$.  Moreover $\Phi=\widetilde{\Phi}\meet h$.
  \item\label{coversone} If $\Phi\nleq{h}$ then it only covers
    $\Phi\meet h$ in $\chi(\uS^n,\uS)_J$.  
  \end{enumerate}
\end{proposition}

\begin{proof}
  (\ref{incomparablecovers}).
  The case $n=1$ is taken care of in Corollary~\ref{StoSj}.
  For $n>1$, Theorem~\ref{thm:ji}.\ref{thm:ji nlt j} tells us that
  either $\Phi=(0,\phi\meet h,\phi)$ or
  $\Phi=(\phi,\phi,\phi\meet h)$, and similarly 
  $\Gamma=(0,\gamma\meet h,\gamma)$ or 
  $\Gamma=(\gamma,\gamma,\gamma\meet h)$ for some
  $\phi,\gamma\in\chi(\uS^{n-1},\uS)_J$ with $\phi,\gamma\nleq h$. 
  Clearly $(0,\phi\meet h,\phi)$ and $(\gamma,\gamma,\gamma\meet h)$ are
  incomparable since $\phi\nleq h$.   By induction 
  $(0,\phi\meet h,\phi)$ and $(0,\gamma\meet h,\gamma)$ are
  incomparable; similarly, $(\phi,\phi,\phi\meet h)$ and
  $(\gamma,\gamma,\gamma\meet h)$ are incomparable.  \\
  (\ref{uniquecovers}).
  Once again, the case $n=1$ is taken care of in
  Corollary~\ref{StoSj}.  
  For  $n>1$, Theorem~\ref{thm:ji}.\ref{thm:ji lt j} tells us that either
  $\Phi=(0,\psi,\psi)$ or $\Phi=(\psi,\psi,\psi)$, for some
  $\psi\in\chi(\uS^{n-1},\uS)_J$ with $\psi\leq h$.  By induction, there
  is a unique $\widetilde{\psi}\in\chi(\uS^{n-1},\uS)_J$ with
  $\widetilde{\psi}\nleq h$, such that $\widetilde{\psi}$ covers
  $\psi$.  Moreover, $\psi=\widetilde{\psi}\meet h$.  By
  Theorem~\ref{thm:ji}.\ref{thm:ji nlt j}, we have that
  $(0,\psi,\widetilde{\psi})$ and 
  $(\widetilde{\psi},\widetilde{\psi},\psi)$ are in
  $\chi(\uS^{n},\uS)_J$.  Clearly, $(0,\psi,\widetilde{\psi})$ covers
  $(0,\psi,\psi)$.  That $(\widetilde{\psi},\widetilde{\psi},\psi)$
  covers $(\psi,\psi,\psi)$ follows from the fact that neither
  $({\psi},\widetilde{\psi},\psi)$ nor
  $(\widetilde{\psi},{\psi},\psi)$ are morphisms by
  Lemma~\ref{construct}.2.d.  This shows the existence part, by taking
  $\widetilde{\Phi}=(0,\psi,\widetilde{\psi})$ when
  $\Phi=(0,\psi,\psi)$ and
  $\widetilde{\Phi}=(\widetilde{\psi},\widetilde{\psi},\psi)$  when 
  $\Phi=(\psi,\psi,\psi)$.  In either case note that
  $\Phi=\widetilde{\Phi}\meet h$.   For uniqueness, assume that
  $\Gamma\in\chi(\uS^{n},\uS)_J$, with $\Gamma\nleq h$, covers
  $\Phi$. By Theorem~\ref{thm:ji}.\ref{thm:ji nlt j}, we have either
  $\Gamma=(0,\gamma\meet h,\gamma)$ or
  $\Gamma=(\gamma,\gamma,\gamma\meet h)$ for some
  $\gamma\in\chi(\uS^{n-1},\uS)_J$ with $\psi\nleq h$. In the first
  case, it follows that $\Phi=(0,\psi,\psi)$ and $\gamma$ covers
  $\psi$.  By uniqueness of $\widetilde{\psi}$, we must have
  $\gamma=\widetilde{\psi}$, and $\Gamma=(0,\psi,\widetilde{\psi})$.
  In the second case, since
  $(0,\psi,\psi)\leq\Gamma$ implies $(\psi,\psi,\psi)\leq\Gamma$ we
  must have $\Phi=(\psi,\psi,\psi)$ and $\gamma$ covers $\psi$. Again,
  by uniqueness of $\widetilde{\psi}$ we get
  $\gamma=\widetilde{\psi}$, and
  $\Gamma=(\widetilde{\psi},\widetilde{\psi},\psi)$. \\
  (\ref{coversone}).
  By part~\ref{incomparablecovers}, $\Phi$ can only cover elements of
  $\chi(\uS^{n},\uS)_J\cap\downset{h}$, and part~\ref{uniquecovers}
  yields the uniqueness. 
  By Corollary~\ref{cor:meet h is ji}, $\Phi\meet h$ is in
  $\chi(\uS^n,\uS)_J$, and 
  by part~\ref{uniquecovers} it has a unique cover
  $\widetilde{\Phi}\nleq h$.  Therefore, we must have
  $\widetilde{\Phi}\leq\Phi$.  Now part~\ref{incomparablecovers} forces 
  $\widetilde{\Phi}=\Phi$, so $\Phi$ covers $\Phi\meet h$.  
\end{proof}

We can now combine the results of Propositions~\ref{prop:thebase}
and~\ref{prop:covers} to describe the poset structure of
$\chi(\uS^{n},\uS)_J$. 

\begin{theorem}\label{thm:hairy cube}
  The poset $\chi(\uS^{n},\uS)_J$ consists of two parts:  the ``{\sl
    base}'' $Y^n$ which is an $n$-cube and the ``{\sl hairs}''
  $\chi(\uS^{n},\uS)_J\setminus Y^n$, which are pairwise incomparable.
  Each element of the base is covered by a unique hair.  Each hair
  covers a unique base element.  
\end{theorem}

In Theorem~\ref{PSSisomorphictohairycube} we show that these partial
order properties of $\chi(\uS^{n},\uS)_J$ completely determine it as a
partially Stone space.  Even though we will only need the fact that
$\chi(\uS^{n},\uS)_J$ is a poset, we can actually see that it is a
meet-semilattice.  

\begin{corollary}\label{meetsemilattice}
  $\chi(\uS^n,\uS)_J$ is a meet-semilattice.  For any
  $\Phi,\Gamma\in\chi(\uS^n,\uS)_J$,  
  \[ \Phi\meet\Gamma = (\Phi\meet h)\meet(\Gamma\meet h) \]
  is an element of the base of the hairy cube. 
\end{corollary}

\subsection{The Partially Stone Space Corresponding to the Hairy
Cube}\label{hairy cube} 

  \ 

  There is a well-known duality between $T_0$ Alexandrov spaces and
  partial orders. For details see~\cite{Johnstone}.  For us, it will
  be more convenient to use the opposite partial order and the
  opposite (interchange open and closed) topology. 

  Here are the details.  Given a poset $P$, the set
  $\Lambda=\{\downset{p}|p\in{P}\}$ forms a basis for a topology on
  $P$; we refer to it as the ``{\sl downset topology}'' (it is the
  opposite of the ``{\sl Alexandrov topology}'').  Given a $T_0$-space
  $X$, the following defines a partial order on $X$: for $x,y\in{X}$,
  set $x\leq y$ if and only if every open subset of $X$ that contains
  $y$ must also contain $x$. We refer to this as the ``{\sl open set
    partial order}'' (it is the opposite of the ``{\sl specialization
    order}'').  

Just like in the Alexandrov duality we get the following lemma:

\begin{lemma}\label{Alexandrov}\ 
  \begin{enumerate}
  \item \label{thesametopology}  
    Suppose that $X$ is a $T_0$ Alexandrov space, partially ordered
    with the open set partial order.  Then the downset topology
    induced by the partial order, and the original topology on $X$ are
    the same.
  \item\label{thesamepartialorder} 
    Suppose that $P$ is a partially ordered set, which is given the
    downset topology.  Then the open set partial order and the original
    partial order are the same.
\end{enumerate}
Moreover, a function between posets is order preserving if and only if it is
continuous as a map between $T_0$ Alexandrov spaces.
\end{lemma}

Note that any Partially Stone Space is $T_0$, and every finite space
is Alexandrov.  
Recall that for a finite partially complemented distributive lattice
$\calL$, the partially Stone space $[X,Y]$ corresponding to $\calL$
under the $\bsr\leftrightarrows\pss$ duality has
$X=\calL_J$ and \mbox{$Y=\{x\in X| x \leq h\}$}. 
From the definition of the Stone topology on $X$ and the previous
lemma we get the following corollary:

\begin{corollary}\label{cor:same partial order}
  Let $\calL$ be a finite partially complemented
  distributive lattice. Let $\calL_J$ be the set of
  join-irreducible elements of $\calL$, with the partial order
  inherited from $\calL$.  Let $[X,Y]$ be the Partially Stone
  Space corresponding to $\calL$ under the duality
  $\bsr\leftrightarrows\pss $.  
  \begin{enumerate}
  \item The downset topology on $\calL_J$ and the Stone topology on
    $X$ are the same.
  \item The partial order in $\calL_J$ and the open set partial
    order in $X$ are the same.
  \end{enumerate}
 \end{corollary}

We now  characterize those partially
Stone spaces that correspond to the Hairy Cube. 

\begin{theorem}\label{PSSisomorphictohairycube} 
  Let $[X,Y]$ be a Partially Stone Space partially ordered with the
  open set partial order such that the following hold:
  \begin{enumerate} 
  \item $Y$ is poset isomorphic to $2^n$;
  \item The elements of $X-Y$ are pairwise incomparable.  
  \item Every $y\in{Y}$ has a unique cover $x\in{X-Y}$;
  \item Every $x\in{X-Y}$ covers only one $y\in{Y}$;
\end{enumerate}

Then $[X,Y]$, and the Partially Stone Space corresponding to
$\chi(\uS^n,\uS)$ under the
$\bsr\leftrightarrows\pss$ duality, are Partially
Stone Space homeomorphic.
\end{theorem}

\begin{proof}
  It is clear from Theorem~\ref{thm:hairy cube} that $X$ and
  $\chi(\uS^{n},\uS)_J$ are poset isomorphic, under an isomorphism
  $\eta:X\to \chi(\uS^{n},\uS)_J$ that maps $Y$ onto $Y^n$. By
  Lemma~\ref{Alexandrov}, $\eta$ is a homeomorphism of topological
  spaces.  The only additional fact needed is that $\eta|Y$
  is coherent.  This follows from the fact that $Y$ and $Y^n$ are
  finite. 
\end{proof}

\subsection{Polynomials}

\ 

Recall from Lemma~\ref{lemma:term functions} that $\chi(\uS^n,\uS)_J$
is just the set of $n$-ary term functions 
on the algebra $\S$ 
which are join-irreducible. The following result shows that these can
be obtained using only the $\meet$ operation and \mbox{unary
  \rule[7pt]{8pt}{0.4pt}} operation introduced in Lemma~\ref{barops}.

When writing term functions for an algebra, projection maps are
usually called ``{\sl variables}'', and denoted by lower case letters.
Given a variable $p_i$, with $1\leq i \leq n$, and
$\epsilon_i\in\{0,1\}$ we define

\[
p_i^{\epsilon_i} = \left\{ \begin{array}{ll}
p_i            &\text{ if } \epsilon_i=0 \\
\overline{p_i} &\text{ if } \epsilon_i=1 \\
\end{array}\right.
\]

\begin{theorem}\label{polynomials}  
  Let $\Phi\in\chi(\uS^n,\uS)_J$.
  \begin{enumerate}
  \item When $\Phi\nleq h$ it can be uniquely written as a polynomial of
    the form 
    \[ \Phi = \bigmeet_{i=1}^n p_i^{\epsilon_i}  \]
  \item When $\Phi\leq h$ it can be uniquely written as a polynomial
    of the form
    \[ \Phi = \left(\bigmeet_{i=1}^n p_i^{\epsilon_i}\right)\meet h  \]
  \end{enumerate}
\end{theorem}

\begin{proof}
  When $n=1$ the statement follows from Corollary~\ref{StoSj}. If $n>1$
  recall the map $\eta$ in Proposition~\ref{prop:thebase}.  We will
  show that taking $\epsilon=\eta(\Phi\meet h)$, proves existence.
  By Theorem~\ref{thm:ji}.\ref{thm:ji both} there is
  $\phi\in\chi(\uS^{n-1},\uS)_J$, and we have two cases to
  consider. In the first case, $\Phi=(0,\phi\meet h,\phi)$, 
  \[ \eta(\Phi\meet h)=\eta(0,\phi\meet h,\phi\meet
  h)=(0,\eta(\phi\meet h)), \text{\ and\ } 
  \Phi=(0,h,1)\meet(\phi,\phi,\phi)
  =p_1\meet\phi
  .\]
  In the second case, $\Phi=(\phi,\phi,\phi\meet h)$, 
  \[ 
  \eta(\Phi\meet h)=\eta(\phi\meet h,\phi\meet h,\phi\meet h)=
  (1,\eta(\phi\meet h)), \text{\ and } 
  \Phi=(1,1,h)\meet(\phi,\phi,\phi)
  =\overline{p_1}\meet\phi
  .\]
  In either case, $\Phi\leq h$ if and only if $\phi\leq h$.  By
  induction we have:
  \[ \text{when }\Phi\nleq h,\quad
  \Phi = p_1^{\epsilon_1}\meet\phi
  =p_1^{\epsilon_1}\meet\bigmeet_{i=2}^n p_i^{\eta(\phi\meet h)} 
  =\bigmeet_{i=1}^n p_i^{\eta(\Phi\meet h)}; \]
  \[ \text{when }\Phi\leq h,\quad
  \Phi = p_1^{\epsilon_1}\meet\phi
  =p_1^{\epsilon_1}\meet\bigmeet_{i=2}^n p_i^{\eta(\phi\meet h)}\meet h
  =\bigmeet_{i=1}^n p_i^{\eta(\Phi\meet h)}\meet h. \]
  Uniqueness follows from the above, the bijectivity of $\eta$ in
  Proposition~\ref{prop:thebase}, and Proposition~\ref{prop:covers}.
\end{proof}

\pagebreak

\newcommand{\Phisub}[3]{\ensuremath #1 \meet #2 \meet #3}

\begin{center}\footnotesize
The ``Hairy Cube'' for $n=3$
\end{center}
\[
\xymatrix @!C@C=-40pt @!R@R=6pt
{
& & & & & \Phisub{\overline{p}_1}{\overline{p}_2}{\overline{p}_3} \ar@{-}[ddl] & & \\ \\
\Phisub{\overline{p}_1}{p_2}{\overline{p}_3} \ar@{-}[ddr] & & & & \Phisub{\overline{p}_1}{\overline{p}_2}{\overline{p}_3}\wedge{h} \ar@{-}[ddlll] \ar@{-}[ddrrr] \ar@{-}[dddd] & & &  & \Phisub{p_1}{\overline{p}_2}{\overline{p}_3} \ar@{-}[ddl] \\ \\
& \Phisub{\overline{p}_1}{p_2}{\overline{p}_3}\wedge{h}
\ar@{-}[dddrrr] \ar@{-}[dddd] & & & &
\Phisub{\overline{p}_1}{\overline{p}_2}{p_3} \ar@{-}[ddl] & & \Phisub{p_1}{\overline{p}_2}{\overline{p}_3}\wedge{h} \ar@{-}[dddd] \ar@{-}[dddlll] \\
& & & \Phisub{p_1}{p_2}{\overline{p}_3} \ar@{-}[ddr] & & \\
\Phisub{\overline{p}_1}{p_2}{p_3} \ar@{-}[ddr] & & & & \Phisub{\overline{p}_1}{\overline{p}_2}{p_3}\wedge{h}  \ar@{-}[ddlll] \ar@{-}[ddrrr] & & & & \Phisub{p_1}{\overline{p}_2}{p_3} \ar@{-}[ddl] \\ 
& & & & \Phisub{p_1}{p_2}{\overline{p}_3}\wedge{h} \ar@{-}[dddd] \\ 
& \Phisub{\overline{p}_1}{p_2}{p_3}\wedge{h} \ar@{-}[dddrrr] & & & & & & \Phisub{p_1}{\overline{p}_2}{p_3}\wedge{h} \ar@{-}[dddlll] \\
& & & \Phisub{p_1}{p_2}{p_3} \ar@{-}[ddr]\\
 \\ 
& & & & \Phisub{p_1}{p_2}{p_3}\wedge{h} & & 
}
\]

\section{Neither Full nor Strong, A Small Strong Structure}
\label{sec:makeitstrong}

From the optimal duality established in Theorem~\ref{optdual}, we
have obtained geometric and polynomial characterizations of
$\chi(\uS^n,\uS)_J$.  However, this duality is neither full nor
strong, as the following result shows.

\begin{theorem}\label{notfs}
The duality yielded by $\uS$ is neither full nor strong.
\end{theorem}

\begin{proof}

First we will show that the duality yielded by $\uS$ on $\calA$ 
is not strong.  Consider the \SSDT~\cite{Clark-Davey}*{3.2.9}.  
Since $\uS$ is a total structure, if it were to yield a strong duality
on $\calA$, it would satisfy the Finite Term
Closure Condition:
\vskip 5pt
FTC: For any $n\in\N$, $X\leqslant{\uS}^n$ and
$y\in{\leqslant{\uS}^n\backslash{X}}$, $\exists$ term functions
$\sigma,\tau : \uS^n \rightarrow \uS$ on $\S$ (that is
morphisms) that agree on $X$ but not at $y$.
\vskip 5pt
Consider $X = \{0,1\}\leqslant{\uS}$ and $y=h$.  Upon viewing the
diagram in Lemma~\ref{StoS}  we see that $(\uS,\uS)|_X =
\{(0,0),(0,h),(0,1),(h,h),(h,1),(1,h),(1,1)\}$ and hence no pair
$\sigma,\tau\in\chi(\uS,\uS)$ can agree on $\{0,1\}$ (and differ at $h$).

To show that the duality yielded on $\calA$ by $\uS$ is not even a
full duality,  consider the \FSDT~\cite{Clark-Davey}*{3.2.4}, which
states that $\uS$ yields a strong duality 
on $\calA$ if and only if $\uS$ yields a full duality on
$\calA$ and $\uS$ is injective in $\chi$.  Since  $\uS$
is injective in $\chi$, by Theorem~\ref{optdual}, it follows
that $\uS$ does not yield a full duality on $\calA$.
\end{proof}

The \DAT~\cite{Clark-Davey}*{1.5.3} establishes embeddings of $X$ into
$DE(X)$ for 
every $X\in{\chi}$. The failure of the duality to be full, and
therefore strong, is in the failure of $X$ to be isomorphic to
$DE(\bfX)=\calA(\chi(\bfX,\uS),\S)$ for every
$X\in{\chi}$.  In order to obtain a strong duality, we need to
add structure to $\uS$ that will eliminate objects of $\chi$
that are a closed substructure of a power of $\uS$ and are not
term/hom-closed.  

The \NUSDT,~\cite{Clark-Davey}*{3.3.8} uses the irreducibility index
of $\S$, defined below, to give an exact recipe for constructing a 
generating structure $\uSNU$ that will yield a strong duality on
$\calA$.  One simply needs to add the all the algebraic n-ary
operations and partial operations for \mbox{$1 \leq$ n $\leq
  \irr(\S)$} to the structure on $\uS$,
to obtain
$\uSNU$.  We refer to this method we refer to as the ``Near Brute
Force'' method. 
One can then apply the methods of~\cite{Clark-Davey}, particularly the
\MSSDL~\cite{Clark-Davey}*{3.2.3},
to work towards obtaining an optimal strong duality.

\begin{definition}\label{irrind}
  Let $A$ be a finite algebra.  The {\sl irreducibility of} $A$ is the
  least $n\in\N$ such that the zero congruence on $A$ is a meet of $n$
  meet-irreducible congruences.  The {\sl irreducibility index} of $A$
  denoted $\irr(A)$, is the maximum of the irreducibilities of
  subalgebras of $A$.
\end{definition}

\begin{lemma}\label{irr2}  
$\irr(\S)$ is 2.
\end{lemma}

\begin{proof}
  From the lattice of subalgebras of $\S^2$ in
  Lemma~\ref{subS2} it is easy to check that {$\con(\S) =
  \{\Delta,r_3,r_2 \cap {r_2}^{-1},\S^2\}$}, and it is
  isomorphic to the 2 dimensional cube.  $\S$ has no
  subalgebra other than itself.
\end{proof}

Recall that a n-ary operation $g$ on $S$ is algebraic over $\S$ if
$g\in\calA(\S^n,\S)$; a n-ary partial operation $h$ on $S$ is
algebraic over $\S$ if $h\in{\calA(X,\S)}$ for some $X\leq{\S^n}$.
Furthermore, these conditions are equivalent to saying that the
corresponding graphs form subalgebras of $\S^{n+1}$.  As can be seen
through the proofs of Lemma~\ref{restproj} and
Proposition~\ref{strong}, 
the number of algebraic binary partial operations on $\S$ is too large, for the
brute force method to yield a useful structure.  We want a structure
$\uSs$ that yields a strong duality on $\calA$ that is as simple as
possible.  On the other hand, the only algebraic binary total
operations on $\S$ are the projections.

\begin{proposition}\label{proj}
  Let $n\in\N$ and
  $\Lambda:\S^n\longrightarrow{\S}$. 
  Then $\Lambda\in{\calA(\S^n,\S)}$ if and
  only if it is a projection map, $\Lambda=\Pi_{i}^n$. 
  for some 
  $1\leq{i}\leq{n}$.  
\end{proposition}

\begin{proof}
  ${\Lambda^{-1}(1)}$ (resp.$\Lambda^{-1}(0)$) is a prime filter
  (resp. ideal) of $\S^n$, hence there is $x$ (resp. $y$)
  join-irreducible (resp. meet-irreducible) such that 
  $\Lambda^{-1}(1)=\upset{x}$
  (resp. $\Lambda^{-1}(0)=\downset{y}$).  
  As $x$ is join-irreducible, it follows that 
  $\Pi_j(x)\neq{0}$ for at most one $j$,  Since $h$ is a 
  constant, $\Lambda(x\meet h)=\Lambda(x)\meet h=1\meet h=h$, and we cannot
  have $x\meet h=x$   and therefore  $\Pi_j(x)=1$.  Similarly,
  $\Pi_i(y)\neq{1}$ for at most one $i$, and $\Pi_i(y)=0$.  
  If $i\neq{j}$, then $x\leq{y}$ and $\Lambda(x)\leq{\Lambda(y)}$
  yielding a  contradiction,  hence $i=j$. Now, let $z\in{X}$ and
  consider the following cases: 
\begin{enumerate}

\item $\Pi_i(z)=0$.  In this case $z\in{\downset{y}}$ and hence $\Phi(z)=0$.

\item $\Pi_i(z)=h$.  In this case $z\notin{\upset{x}}$ and
  $z\notin{\downset{y}}$,  hence $\Phi(z)=h$.

\item $\Pi_i(z)=1$.  In this case $z\in{\upset{x}}$ and hence $\Phi(z)=1$.
 
\end{enumerate}

So we see that $\Phi=\Pi_{i}$. 
The converse holds as projections are homomorphisms. 
\end{proof}

As a result of Proposition~\ref{proj} the only total operations in
$\uSNU$ are $\Pi^1_1$, $\Pi^2_1$ and $\Pi^2_2$.  There are no proper
algebraic unary partial operations, since the only subalgebra of $\S$
is $\S$ itself, but the set of algebraic binary partial operations is
too large to be useful.

In order to get a manageable structure $\uSs$ that yields a strong
duality on $\calA$, we will reduce the set of algebraic binary partial
operations using the \MSSDL.  Any structure that strongly entails
$\uSNU$ will also yield a strong duality on $\calA$.  To get such
structure $\uSs$, we may delete from $\uSNU$ those partial operations
that are restrictions of other total or partial operations left in the
structure. In particular, we may delete any partial operation which is
a restriction of a projection.

Unlike total algebraic operations, which by Proposition~\ref{proj}
have to be projections, for the algebraic binary partial operations
there is a little more room as the following lemma shows. 

\begin{lemma}\label{restproj}
  Let $A\leq{\S^2}$. For $i=1, 2$, let $L_i=\Pi_i^2$
  restricted to $\downset{h}$, $U_i=\Pi_i^2$ restricted to $\upset{h}$.
  Let $\Lambda\in{\calA(A,\S)}$.  

  \begin{enumerate}
  \item If $A\cap{\upset{h}}$ contains $(h,1)$ or $(1,h)$ then
    $\upset{h}\subseteq{A}$ and $\Lambda$ restricted to $\upset{h}$
    equals  $U_1$ or $U_2$.
  \item $\Lambda$ restricted to $A\cap{\downset{h}}$ equals $L_1$ or
    $L_2$.  
  \end{enumerate} 
\end{lemma}

\begin{proof}
  If $A\cap{\upset{h}}$ contains either $(h,1)$ or $(1,h)$, applying the
  complement operation from Lemma~\ref{barops} we get the other and hence
  $\upset{h}\subseteq{A}$.  The rest follows from the facts that
  $\Lambda$ is order preserving and every element of $\Delta_S$ is a
  constant.  
\end{proof}

It is easy to check that the binary partial operation $\lambda_1$ with
domain $r_1$ having graph
\[\gr(\lambda_1)=\{(0,0,0),(h,h,h),(1,1,1),(0,h,h),(0,1,h),(h,0,0),(h,1,h),(1,h,1)\},\]
is in fact algebraic. It combines $U_1$ and $L_2$.   Similarly, the
combination of  $U_2$ and $L_1$ yields the
algebraic binary partial operation $\lambda_2$ with domain $r_1^{-1}$ having graph
\[\gr(\lambda_2)=\{(0,0,0),(h,h,h),(1,1,1),(0,h,0),(h,0,h),(1,0,h),(h,1,1),(1,h,h)\}.\]
These two partial operations are examples of algebraic binary partial
operations which are not restrictions of projections.
Lemma~\ref{restproj} places constraints on such operations to be
combinations  of $L_i$ and $U_k$ with $i \neq k$.  
In a sense, $\lambda_1$ and $\lambda_2$ are the only such examples, as
it is more clearly stated in the proof of the next proposition.

\begin{proposition}\label{strong}
  Let $n\in\N$ and $\Pi_i^{n}:\uS^n\longrightarrow{\uS}$
  denote the $i$-th projection map for any $i\in\N$ with
  $1\leq{i}\leq{n}$ and $\Tau$ denote the discrete topology,
  then the structure 
  \[
  \uSs=<S;\{r_1,r_2,r_3\},\{\Pi_1^2,\Pi_2^2\},\{\lambda_1,\lambda_2\},\Tau>
  \] 
  yields a strong duality on $\calA$.  
\end{proposition}

\begin{proof}

  As $\Pi_1^1$ is the identity map on $S$, it has no effect on any
  topological category generated by a structure with $S$ as its
  carrier set. Therefore,
  we do not need to include it in the list of total operations.
  By the \MSSDL\ we only need to show that any algebraic binary partial
  operation in the structure $\uSNU$ is a restriction of either a
  projection or one of $\lambda_1$, $\lambda_2$. 
  From Lemma~\ref{restproj}, we see that the only homomorphisms which are
  not restrictions of projections must consist of either
  a combination of $L_1$ and $U_2$ or a combination of $L_2$ and
  $U_1$.
  Let us first consider
  the subalgebra $A=r_1\leq \S^2$ which contains the element
  $(0,1)$. Let $\lambda:A\to \S$ be a homomorphism which is not a
  restriction of a projection.  
  If we had $\lambda(0,1)=0$ this would force 
  $\lambda(0,h)=\lambda(h,h)\meet\lambda(0,1)=h\meet 0=0$ and
  by Lemma~\ref{restproj}, $\lambda$ restricted to $A\cap{\downset{h}}$
  would have to equal $L_1$. Moreover, we would have
  $\lambda(h,1)=\lambda(h,h)\join\lambda(0,1)=h\join 0=h$, and
  $\lambda$ restricted to $A\cap{\upset{h}}$ would have to
  equal $U_1$, making
  $\lambda$ a restriction of $\Pi_1^2$. Similarly, if we had
  $\lambda(0,1)=0$ this would force $\lambda$ to be a restriction of
  $\Pi_2^2$.  Therefore, we must have $\lambda(0,1)=h$. As above, this
  forces 
  $\lambda(0,h)=\lambda(h,h)\meet\lambda(0,1)=h\meet h=h$, and
  $\lambda(h,1)=\lambda(h,h)\join\lambda(0,1)=h\join h=h$, making
  $\lambda$ a combination of $L_2$ and $U_1$, i.e. $\lambda_1$.   So,
  $\lambda_1$ is the only partial algebraic operation with domain
  $r_1$ which is not a restriction of a projection.
  A similar argument shows that $\lambda_2$ is the only partial
  algebraic operation with domain $r_1^{-1}$, which is not a restriction of
  a projection.

  Note that the argument above does not make use of the fact that
  $(h,0)$ is in $r_1$.  Therefore, it also shows that the only partial
  algebraic operation with domain $r_2$ which is not a restriction of
  a projection, must be the restriction of $\lambda_1$.  Similarly,
  the only partial algebraic operation with domain $r_2^{-1}$ which is not
  a restriction of a projection, must be the restriction of
  $\lambda_2$.

  As shown in Lemma~\ref{subS2} any other subalgebra $A\leq\S^2$ must
  be a subalgebra of $r_1\cap r_1^{-1}$.  Hence, by
  Lemma~\ref{restproj}, any partial algebraic operation with domain
  $A$, which is not a restriction of a projection,  
  must be a restriction of either $\lambda_1$ or $\lambda_2$. 
\end{proof}

As an intermediate step towards an optimal strong duality, 
we will show that we can eliminate the total operations and $\lambda_2$ from
$\uSs$ and still achieve a strong duality.  First we need the
following definitions:

\begin{definition}\label{epclone}\cite{Clark-Davey}
A set of partial operations on a set $X$ is called an {\sl enriched
partial clone} if it is closed under composition and contains
$\{\Pi_i^n\|1\leq{i}\leq{n},n\in\N\}$.  The enriched
partial clone generated by a set of partial operations $H$ on a set $X$
is the smallest enriched partial clone containing $H$ and is denoted $[H]$.
Given a structured topological space 
$\bfX=\langle X;R,G,H,\Tau \rangle$, the enriched partial clone of 
$\bfX$ is the smallest enriched partial clone on $X$ containing
$G\cup{H}$ and is denoted $[G\cup{H}]$.  Furthermore, let $\mathcal{P}$ 
denote the set of all finitary algebraic, partial or
total operations on $\S$.  Then $\mathcal{P}$ is an enriched
partial clone on $S$ and is referred to as the {\sl enriched
partial hom-clone} of $\S$.    
\end{definition}

\begin{definition}\label{homent}\cite{Clark-Davey}
Let $P\subseteq{\mathcal{P}}$ and $k\in{\mathcal{P}}$. We say that $P$
{\sl hom-entails} $k$ if, for all non-empty sets $\Omega$, each topologically
closed subset of $S^{\Omega}$ which is closed under the partial
operations in $P$ is also closed under $k$.  Define
$\overline{P}=\{k\in{\mathcal{P}}|P$ hom-entails $k\}$.  Then
$P\longmapsto{\overline{P}}$ is a closure operator on $\mathcal{P}$
and we refer to $\overline{P}$ as the hom-entailment closure of $P$.
\end{definition}

\begin{corollary}\label{smallstrong}
The structure
$\uSsp=\langle S;\{r_1,r_2,r_3\},\{\lambda_1\},\Tau\rangle$ 
yields a strong duality on $\calA$.
\end{corollary}

\begin{proof}
  Let $G=\{\Pi_1^2,\Pi_2^2\}$ and $H=\{\lambda_1,\lambda_2\}$, then by
  Lemma~\ref{strong} and the \BFSDT~\cite{Clark-Davey}*{3.2.2},
  $\overline{G\cup{H}}=\mathcal{P}$, i.e. $G\cup{H}$ hom-entails every
  finitary algebraic partial or total operation on $\S$.  Now consider
  Definition~\ref{epclone}, and note that $[H]=[G\cup{H}]$.  Further
  note that $\lambda_1(\Pi_2^2,\Pi_1^2)=\lambda_2$, and we see that
  $[\lambda_1]=[H]$.  Let $k:A\longrightarrow{S}$ be any element of
  $\mathcal{P}$ and consider what we refer to as the
  \TALM~\cite{Clark-Davey}*{9.4.1}.  Parts (i.a) and (i.c) show
  $\overline{\{\lambda_1\}}=\mathcal{P}$ and hence the structure
  $\uSsp$ yields a strong duality on $\calA$.
\end{proof}

Now we wish to show that the partial operation $\lambda_1$ entails $r_1$
and $r_3$, for that we need the following Lemma:

\begin{lemma}\label{entail}
  Let $\uS'=\langle S;\{\lambda_1\},\Tau \rangle$ with $\lambda_1$ as
  the only partial operation.
  \begin{enumerate}

  \item\label{entail part 1} Let $\bfX\in\ScP(\uS')$ and
    $g:\bfX\longrightarrow{\uS'}$ be a morphism.  Then $g$ preserves
    $r_1$ and $r_3$, and therefore $\{\lambda_1\}$ entails $r_1$ and
    $r_3$.

  \item\label{entail part 2} Now, let $\Phi:\uS'\longrightarrow{\uS'}$
    be given by $\Phi=(h,0,0)$. Then $\Phi$ preserves $\lambda_1$ but
    not $r_2$, and therefore $\uS'$ does not yield a duality on
    $\calA$.

  \end{enumerate}
\end{lemma}

\begin{proof}
  (\ref{entail part 1})
  Let $(Y,Z)\in{r_1^{\bfX}}=\dom(\lambda_1^\bfX)$. As $g$ preserves
  $\lambda_1$, $(g(Y),g(Z))\in{\dom(\lambda_1^{\uS'})}=r_1^{\uS'}$ and
  $g$ preserves $r_1$.
  Now let $(Y,Z)\in{r_3^{\bfX}}\subseteq{\dom(\lambda_1^\bfX)}$, and
  consider $(Y,Z,\lambda_1(Y,Z))$.  For any $i$,
  $(Y_i,Z_i)\in{r_3^{\uS'}}$ hence $Z_i=\lambda_1(Y_i,Z_i)$ and
  therefore $\lambda_1(Y,Z)=Z\in\bfX$.  Moreover,
  $g(Z)=g(\lambda_1(Y,Z))$ which implies that
  $(g(Y),g(Z))\in{r_3^{\uS'}}$.
  The fact that $\{\lambda_1\}$ entails $r_1$ and $r_3$ follows from
  the \EL.  

  (\ref{entail part 2})
  The fact that $\Phi$ preserves $\lambda_1$ is easily seen by
  inspection of $\Phi(\lambda_1(x,y))$ and
  $\lambda_1(\Phi(x),\Phi(y))$ for each
  $(x,y)\in{\dom(\lambda_1)}$.  Now note that
  $(\Phi(0),\Phi(h))=(h,0)$ and hence $\Phi$ fails to preserve $r_2$.
  From Lemmas~\ref{StoS} and~\ref{lemma:term functions} we see that
  $\uS'$ does not yield duality on $\calA$.  
\end{proof}

We now have enough results to prove the major Theorem of this section.

\begin{theorem}
Let $\Tau$ denote the discrete topology and
$\uSos=\langle S;\{r_2\},\{\lambda_1\},\Tau\rangle$, then $\uSos$ yields an
optimal strong duality on $\calA$.
\end{theorem}

\begin{proof}
The fact that $\uSos$ yields a strong duality on $\calA$ follows from
Theorem~\ref{optdual}, Corollary~\ref{smallstrong},
Lemma~\ref{entail}.\ref{entail part 1} and the \MSSDL.
The fact that it is optimal follows from
Lemma~\ref{entail}.\ref{entail part 2} and
Theorem~\ref{optdual}.
\end{proof}

Combining this theorem with Lemma~\ref{lemma:term functions} we get
the final result. 

\begin{corollary}\label{persist}
  Let $\chi_{os}=\ScP(\uSos)$, then for any $n\in\N$,
  $\chi_{os}(\uSosn,\uS)_J$  is the n-dimensional Hairy Cube. 
\end{corollary}

\bibsection{\refname}


\begin{biblist}

\bib{Clark-Davey}{book}{
   author={Clark, David M.},
   author={Davey, Brian A.},
   title={Natural dualities for the working algebraist},
   series={Cambridge Studies in Advanced Mathematics},
   volume={57},
   publisher={Cambridge University Press},
   place={Cambridge},
   date={1998},
   pages={xii+356},
   isbn={0-521-45415-8},
   review={\MR{1663208 (2000d:18001)}},
}

\bib{Clouse}{thesis}{
   author={Clouse, Daniel J.},
   title={A Dual Representation of Boolean Semirings in a Category of
Structured Topological Spaces},
   date={2002},
   organization={Binghamton University},
   type={Ph.D. dissertation},
}


\bib{Eilenberg}{book}{
   author={Eilenberg, Samuel},
   title={Automata, languages, and machines. Vol. A},
   note={Pure and Applied Mathematics, Vol. 58},
   publisher={Academic Press [A subsidiary of Harcourt Brace Jovanovich,
   Publishers], New York},
   date={1974},
   pages={xvi+451},
   review={\MR{0530382 (58 \#26604a)}},
}

\bib{Guzman}{article}{
   author={Guzm{\'a}n, Fernando},
   title={The variety of Boolean semirings},
   journal={J. Pure Appl. Algebra},
   volume={78},
   date={1992},
   number={3},
   pages={253--270},
   issn={0022-4049},
   review={\MR{1163278 (93d:08007)}},
}

\bib{Johnstone}{book}{
   author={Johnstone, Peter T.},
   title={Stone spaces},
   series={Cambridge Studies in Advanced Mathematics},
   volume={3},
   publisher={Cambridge University Press},
   place={Cambridge},
   date={1982},
   pages={xxi+370},
   isbn={0-521-23893-5},
   review={\MR{698074 (85f:54002)}},
}


\bib{Kuich-Salomaa}{book}{
   author={Kuich, Werner},
   author={Salomaa, Arto},
   title={Semirings, automata, languages},
   series={EATCS Monographs on Theoretical Computer Science},
   volume={5},
   publisher={Springer-Verlag},
   place={Berlin},
   date={1986},
   pages={v+374},
   isbn={3-540-13716-5},
   review={\MR{817983 (87h:68093)}},
}

\bib{Stone}{article}{
   author={Stone, M. H.},
   title={Applications of the theory of Boolean rings to general topology},
   journal={Trans. Amer. Math. Soc.},
   volume={41},
   date={1937},
   number={3},
   pages={375--481},
   issn={0002-9947},
   review={\MR{1501905}},
}

\end{biblist}
\end{document}